\documentclass[12pt]{amsart}
\usepackage{amsfonts, amsmath, amsthm, amssymb, amscd, amsrefs}

\usepackage{hyperref}
\usepackage{graphicx}
\usepackage[all]{xy}
\usepackage{tikz} 
\usepackage{color}
\usepackage{mathrsfs} 

\theoremstyle{definition}

\theoremstyle{definition}    \newtheorem{dfn}{Definition}[section]
\theoremstyle{definition}              
\theoremstyle{definition}    
\theoremstyle{definition}    \newtheorem{prp}[dfn]{Proposition}
\theoremstyle{definition}    \newtheorem{thm}{Theorem}
\theoremstyle{definition}    \newtheorem{lem}[dfn]{Lemma}
\theoremstyle{definition}    
\theoremstyle{definition}    
\theoremstyle{definition}    \newtheorem{nta}[dfn]{Notation}
\theoremstyle{definition}    
\theoremstyle{definition}    
\theoremstyle{definition}    
\newcommand{\trm}{\textrm}

\newcommand{\C}{{\mathbb C}}
\newcommand{\R}{{\mathbb R}}

\newcommand{\Z}{{\mathbb Z}}

\newcommand{\E}{{\mathbb E}}

\newcommand{\cF}{\mathcal F}  
 \newcommand{\cI}{\mathcal I}

\newcommand{\cR}{\mathcal R} 
  
\newcommand{\f}{\mathfrak}

\newcommand{\al}{\alpha}
\newcommand{\be}{\beta}
\newcommand{\ga}{\gamma}     \newcommand{\Ga}{{\Gamma}}
\newcommand{\de}{\delta}

\newcommand{\te}{\theta}    
\newcommand{\si}{\sigma}     
\newcommand{\om}{\omega}     \newcommand{\Om}{{\Omega}}

\newcommand{\Ima}{\textrm{Im }}

\newcommand{\ch}{\textrm{ch}} 
\newcommand{\cs}{\textrm{cs}}

\newcommand{\we}{\wedge}      
\newcommand{\tsr}{\otimes}

\newcommand{\srl}{\stackrel}
 \newcommand{\ra}{\rightarrow} \newcommand{\lra}{\longrightarrow} \newcommand{\embed}{\hookrightarrow}

  \newcommand{\wbar}{\overline}
\newcommand{\wht}{\widehat}
\newcommand{\wtl}{\widetilde}

\newcommand{\ptl}{\partial}
\newcommand{\na}{\nabla}

\newcommand{\bl}{\bullet}
\newcommand{\isom}{\cong}     
\newcommand{\bmat}{\left(\begin{array}}  \newcommand{\emat}{\end{array}\right)}
\newcommand{\barr}{\begin{array}}  \newcommand{\earr}{\end{array}}
\newcommand{\bcd}{\begin{CD}}  \newcommand{\ecd}{\end{CD}}
\newcommand{\beq}{\begin{equation}\begin{aligned}}  \newcommand{\eeq}{\end{aligned}\end{equation}}
\newcommand{\beqs}{\begin{equation*}\begin{aligned}}  \newcommand{\eeqs}{\end{aligned}\end{equation*}}
\title{A smooth variant of Hopkins-Singer differential $K$-theory}
\author{Byungdo Park}
\address{Hausdorff Research Institute for Mathematics (HIM), Poppelsdorfer Allee 45, 53115 Bonn, Germany}
\curraddr{School of Mathematics, Korea Institute for Advanced Study,
85 Hoegiro, Dongdaemun-gu, Seoul, 02455, Republic of Korea}
\email{byungdo@kias.re.kr}
\date{\today}

\begin{document}
\sloppy
\maketitle
\begin{abstract}
We introduce a smooth variant of the Hopkins-Singer model of differential $K$-theory. We prove that our model is naturally isomorphic to the Hopkins-Singer model and also to the Tradler-Wilson-Zeinalian model of differential $K$-theory.
\end{abstract}
\maketitle
\setcounter{tocdepth}{1}
\tableofcontents

\section{Introduction}

A differential cohomology theory is a construction on smooth manifolds combining topological data and differential form data in a homotopy theoretic way. The first construction of a differential cohomology theory was due to Cheeger and Simons \cite{CS} for singular cohomology theory which has applications to geometry. For $K$-theory, Karoubi \cite{Ka} developed $K$-theory with $\R/\Z$ coefficients and Lott \cite{Lo} developed $\R/\Z$-index theory leading to a construction of differential $K$-theory and index theorems in differential $K$-theory. (See \cite{Kl, FL}.) Furthermore, there have been a considerable interest from type IIA/B string theory to represent Ramond-Ramond fields and to formulate $T$-duality. (See \cite{F, KV}.)

In \cite{HS} Hopkins and Singer explicitly constructed a differential cohomology theory for any generalized cohomology theory and hence for $K$-theory. Their construction provides a correct model of differential $K$-theory in the sense of the aforementioned homotopy theoretic idea. Following this work, several authors have developed models of differential $K$-theory by using more geometric cocycle data. (See \cite{BS1, SS, Kl, FL, BSSW, TWZ1, HMSV, TWZ2, GL}.) Furthermore, the Hopkins-Singer model has been revisited by \cite{BG, BNV} with the idea that differential cohomology theories are $\infty$-sheaves of spectra. More recently Grady and Sati \cite{GS} developed spectral sequences in differential generalized cohomology theories and have opened venues in computational aspects.

One natural question arising at this point is whether all the known models of differential $K$-theory are isomorphic. Bunke and Schick \cite{BS2} gave an answer to this question: Any two differential extensions with integration of the same generalized cohomology theory that satisfies certain conditions (such as being rationally even) are uniquely determined up to a unique natural isomorphism. However, it is still interesting to see a direct map between any two different models and proving such a map being an isomorphism has technically intricate aspects.

This paper is a technical report introducing a smooth variant of the Hopkins-Singer differential $K$-theory. This model has an advantage that its cocycles consist of smooth data; continuous maps and singular cochains in the Hopkins-Singer model are replaced by smooth maps and differential forms, respectively. Furthermore, it constitutes an abelian group naturally isomorphic to the original Hopkins-Singer model. Such an aspect facilitates comparisons with other models; we establish a natural isomorphism from the Tradler-Wilson-Zeinalian differential $K$-theory \cite{TWZ1, TWZ2} to the Hopkins-Singer differential $K$-theory, and hence adding one more item to the following list of known direct comparisons between differential $K$-theory models.

\begin{itemize}
\item Freed-Lott-Klonoff model to Hopkins-Singer model: Klonoff \cite[Theorem 4.34]{Kl} (Even) Freed and Lott \cite[Proposition 9.21]{FL}  (Odd).
\item Simons-Sullivan model to Freed-Lott-Klonoff model: Simons and Sullivan \cite{SS}, Ho \cite[Theorem 1]{Ho1}. (Even)
\item Bunke-Schick model to Freed-Lott-Klonoff model: Ho \cite[Proposition 3.2]{Ho2}. (Even)
\item Tradler-Wilson-Zeinalian model to Simons-Sullivan model: Tradler, Wilson, and Zeinalian \cite[Remark 3.27]{TWZ2}. (Even)
\item Tradler-Wilson-Zeinalian model to Hekmati-Murray-Schlegel-Vozzo model: Hekmati, Murray, Schlegel, and Vozzo \cite[Theorem 4.2]{HMSV}. (Odd)
\item Gorokhovsky-Lott model to Freed-Lott-Klonoff model: Gorokhovsky and Lott \cite[Theorem 1]{GL}. (Even and odd)
\end{itemize}

This paper is organized as follows. Section \ref{SEC.MainResults} outlines definitions and theorems in this paper. Section \ref{ProofofTheorem1} proves Theorem \ref{THM.theorem.1} establishing a natural isomorphism from the smooth variant of the Hopkins-Singer model to the Hopkins-Singer model. Section \ref{ProofofTheorem2} proves Theorem \ref{THM.theorem.2} which constructs a natural isomorphism from the Tradler-Wilson-Zeinalian model to the smooth variant of the Hopkins-Singer model. Appendix \ref{APP.A} gives a proof of relative de Rham theorem which is used in Section \ref{sec.injectivity}.

\noindent \textbf{Acknowledgements.} We would like to thank Mahmoud Zeinalian for many helpful discussions and sharing ideas. We also would like to thank Scott Wilson for help in completing this work and useful conversations. We also thank Scott Wilson, Arthur Parzygnat, and Corbett Redden for reading the preliminary version of this paper and providing helpful suggestions and comments. We gratefully acknowledge a partial support on this work from Mahmoud Zeinalian's NSF grant DMS-1309099. Some part of this work was completed during our attendance at the conference \emph{Topology of Manifolds --- A conference in honour of Michael Weiss' 60th birthday}. We thank the organizers for financial support and hospitality. Finally, we thank the anonymous referee for helpful comments.

\section{Main results}\label{SEC.MainResults}
\begin{nta}\label{NTA.1} Throughout this paper, $\cR$ is the real vector space $$K^\bl(\text{point})\tsr \R=\R[u,u^{-1}],$$ where $u$ and $u^{-1}$ are Bott elements with degree $2$ and $-2$, respectively. Also, $X$ is a smooth manifold, $\Om^k(X)$ the differential graded algebra of $\cR$-valued differential $k$-forms on $X$, $\Om_{\trm{cl}}^k(X)$ (resp. $\Om^k_{\trm{exact}}(X)$) the subalgebra of closed (resp. exact) $\cR$-valued differential $k$-forms on $X$, $C^k(X;\cR)$ the degree $k$ smooth singular cochain group of $X$ with coefficients in $\cR$, and $Z^k(X;\cR)$ the subgroup of $C^k(X;\cR)$ consisting of degree $k$ cocycles. Note that our differential forms and cochains are graded by their total degree; for example, $\om\cdot u\in \Om^{k+2}(X)$ if $\om$ is a degree $k$ real-valued differential form. Unless otherwise mentioned, the integration symbol $\int$ means the integration of differential forms over smooth singular chains. We denote by $I$ the closed unit interval $[0,1]$, $p$ the projection $p:X\times I\ra X$ onto the first factor, and $p_i$ the projection onto the $i^{\trm{th}}$ factor of the domain. We will also use a notation $\psi_t$ to denote the $t$-slice maps $\psi_t:X\embed X\times I$ defined by $\psi_t(x)=(x,t)$.
\end{nta}

\begin{nta}\label{NTA.2}
In this paper $\bl$ is always $0$ or $1$. We will use the notation $\cF_\bl$ to denote classifying spaces of complex $K$-theory $\cF_0=BU\times \Z$ and $\cF_1=U$. (We refer readers to Tradler, Wilson, and Zeinalian \cite{TWZ2}, Section 3 for the models of $BU\times\Z$ and $U$ that we will be using.) We also denote by $c_\bl$ the universal Chern character form ${c_\bl}\in\Om_{\trm{cl}}^{\bl}(\cF_\bl)$ representing universal Chern characters in $H^{\bl}(\cF;\cR)$, defined by  ${c_0}:=\ch(\na_{\trm{univ}}):=e^{\frac{i}{2\pi}\na_{\trm{univ}}^2\cdot u^{-1}}$, where $\na_{\trm{univ}}$ is the universal connection on the universal bundle $\E\ra \cF_0$ and ${c_1}:=\text{tr}\sum_{n\in \Z\cup\{0\}}\frac{(-1)^n}{(2\pi i)^{n+1}}\frac{n!}{(2n+1)!} \te^{2n+1}\cdot u^{-2n}$, where $\te$ is the canonical 1-form on the stabilized unitary group $U$ valued in Lie algebra of $U$. (Compare \cite[Definition 3.2]{TWZ2}.)

The space $\cF_\bl$ is endowed with a homotopy commutative $H$-space structure $m_\bl:\cF_\bl\times\cF_\bl \ra  \cF_\bl$. The map $m_0$ is defined in \cite[Definition 3.21]{TWZ2} and $m_1$ in \cite[Definition 3.7]{TWZ2} Both of these maps are smooth by construction. (See also \cite[p.535]{TWZ2} for discussions on smoothness and differential forms on $U$ and $BU\times \Z$.) We will write $\cI_\bl\in\cF_\bl$ to denote $\cI_0:\C_{-\infty}^{\infty}\ra \C_{-\infty}^{\infty}$ the orthogonal projection onto $\C_{-\infty}^{0}$, where $\C_{-\infty}^{\infty}$ and $\C_{-\infty}^{0}$ are $\C$-vector spaces spanned by $\{e_i\}_{i\in\Z}$ and $\{e_i:i\in\Z^{-}\}$, respectively, and $\cI_1:=\textbf{1}$ in $U$.
\end{nta}

\begin{dfn}\label{DEF.HS.diff.Ktheory} Let $c:=c_\bl$ be a universal Chern character form. The \textbf{Hopkins-Singer differential $K$-theory} of $X$, denoted by $\widehat K^\bl(X)$, is an abelian group whose elements are equivalence classes of triples $(f,h,\om)$ consisting of the following data.

\begin{itemize}
\item A continuous map $f:X\ra \cF_\bl$.
\item $\om\in\Om^{\bl}(X)$
\item $h\in C^{\bl-1}(X;\cR)$, satisfying \beq\label{EQN.diff.func.rel.2} \de h=f^*\int c-\int\om.\eeq
\end{itemize} Two triples $(f_0,h_0,\om_0)$ and $(f_1,h_1,\om_1)$ are equivalent if and only if the following holds.\begin{itemize}
\item $\om_0=\om_1$.
\item There exists an \textbf{interpolating triple} $(F,H,p^*\om_0)$ consisting of a continuous map $F:X\times I\ra \cF_\bl$ and a cochain $H\in C^{\bl-1}(X\times I;\cR)$, satisfying $$\de H=F^*\int c-\int p^*\om_0,$$
and whose restriction to $X\times \{0\}$ and $X\times \{1\}$ are the triples $(f_0,h_0,\om_0)$ and $(f_1,h_1,\om_1)$, respectively.
\end{itemize}

The group structure is defined as follows. \beq\label{eqn.HS.triple.sum}(f_1,h_1,\om_1)+(f_2,h_2,\om_2):=(m_\bl(f_1,f_2),h_1+h_2,\om_1+\om_2).\eeq
\end{dfn}

\begin{lem}\label{LEM.group.structure.on.HS}
The addition $+$ defined in \eqref{eqn.HS.triple.sum} is well-defined and gives $\widehat K^\bl(X)$ the structure of an abelian group.
\end{lem}
\begin{proof}
Since $m_\bl(f_1,f_2)^*\int c=f_1^*\int c+f_2^*\int c$, it is readily seen that the RHS of \eqref{eqn.HS.triple.sum} satisfies the triple relation. Suppose $(F,H,p^*\om_1)$ is an interpolating triple between $(f_1,h_1,\om_1)$ and $(f_1',h_1',\om_1)$. The triple $(m_\bl(F,p^*f_2),H+p^*h_2,p^*\om_1+p^*\om_2)$ interpolates between $(m_\bl(f_1,f_2),h_1+h_2,\om_1+\om_2)$ and $(m_\bl(f_1',f_2),h_1'+h_2,\om_1+\om_2)$.

The operation $+$ being an abelian group operation follows from Lemmas 3.9, 3.23, and 3.24 (2) of \cite{TWZ2} and verifying it needs several lemmas we shall prove in the following section. We shall give a proof in Section \ref{SEC.GroupHomo.Thm1}.
\end{proof}

Instead of smooth singular cochains and continuous maps in Definition \ref{DEF.HS.diff.Ktheory}, the following definition uses differential forms and smooth maps.

\begin{dfn} Let $c:=c_\bl$ be a universal Chern character form. The \textbf{smooth Hopkins-Singer differential $K$-theory} of $X$, denoted by $\check K^\bl(X)$, is an abelian group whose elements are equivalence classes of triples $(f,h,\om)$ consisting of the following data. \begin{itemize}
\item A smooth map $f:X\ra \cF_\bl$.
\item $\om\in\Om^{\bl}(X)$
\item $h\in \Om^{\bl-1}(X)$, satisfying \beq\label{EQN.diff.func.rel.1} dh=f^*c-\om.\eeq
\end{itemize}
Two triples $(f_0,h_0,\om_0)$ and $(f_1,h_1,\om_1)$ are equivalent if and only if the following holds. \begin{itemize}
\item $\om_0=\om_1$.
\item There exists an \textbf{interpolating triple} $(F,H,p^*\om_0)$, consisting of a smooth map $F:X\times I\ra \cF_\bl$, and a differential form $H\in \Om^{\bl-1}(X\times I)$, satisfying $$dH=F^*c-p^*\om_0,$$
and whose restriction to $X\times \{0\}$ and $X\times \{1\}$ are the triples $(f_0,h_0,\om_0)$ and $(f_1,h_1,\om_1)$, respectively.
\end{itemize} The group structure is defined as follows. \beq\label{eqn.smooth.HS.triple.sum}(f_1,h_1,\om_1)+(f_2,h_2,\om_2):=(m_\bl(f_1,f_2),h_1+h_2,\om_1+\om_2).\eeq
\end{dfn}

\begin{lem}\label{LEM.group.structure.on.smoothHS}
The addition $+$ defined in \eqref{eqn.smooth.HS.triple.sum} is well-defined and gives $\check K^\bl(X)$ the structure of an abelian group.
\end{lem}
\begin{proof} The RHS of \eqref{eqn.smooth.HS.triple.sum} satisfying the triple axiom follows from $m(f_1,f_2)^*c=f_1^*c+f_2^*c$, and the well-definedness is verified by the same argument as in the proof of Lemma \ref{LEM.group.structure.on.HS}. We give a proof that the operation $+$ is an abelian group operation in Section \ref{SEC.GroupHomo.Thm1}.
\end{proof}

\begin{thm}\label{THM.theorem.1}
Let $c:=c_\bl$ be a universal Chern character form. The assignment \beq\label{EQN.map.in.theorem.1} \check K^\bl(X) &\ra \widehat K^\bl(X)\\
[(f,h,\om)] &\mapsto \left[(f,\int h,\om)\right]
\eeq is an isomorphism of abelian groups that is natural in $X$.
\end{thm}
\begin{proof}
See Section \ref{ProofofTheorem1}.
\end{proof}

\begin{dfn} Let $c:=c_\bl$ be a universal Chern character form and $f_0,f_1:X\ra \cF_\bl$ homotopic smooth maps via a smooth homotopy $F:X\times I\ra \cF_\bl$. The \textbf{Chern-Simons form} of $F$ is \beq\label{EQN.cs.of.htpy} \cs(F):=(-1)^{\bl-1}\int_I \ch(F),\eeq where $\ch(F):=F^*c$. Here $\int_I$ denotes the integration along $I$.
\end{dfn}

\begin{dfn}\label{DEF.CS.equiv} Two smooth maps $f_0,f_1:X\ra \cF$ are \textbf{cs-equivalent} and denoted by $f_0 \sim_{\cs} f_1$ if there exists a smooth homotopy $F:X\times I\ra \cF$ between $f_0$ and $f_1$ such that $\cs(F)\in \Om^*_{\trm{exact}}(X)$.
\end{dfn}
\begin{prp} $\sim_{\cs}$ is an equivalence relation.
\end{prp}
\begin{proof} $f\sim_{\cs} f$ follows from $\cs(f\circ p)=0$. Suppose $f_0\sim_{\cs}f_1$ with homotopy $F$ between $f_0$ and $f_1$ such that $\cs(F)$ is exact. Define $G(x,t):=F(x,1-t)$. Then $\cs(G)=(-1)^{\bl-1}\int_I\ch(G)=-(-1)^{\bl-1}\int_I\ch(F)=-\cs(F)$, which is exact. Hence $f_1\sim_{\cs}f_0$. Finally, suppose $f_0\sim_{\cs}f_1$ through homotopy $F$, and $f_1\sim_{\cs}f_2$ through homotopy $G$. Define \beqs H(x,t):=\left\{ \begin{array}{ll}
F(x,2t)&\trm{if } t\in [0,\frac{1}{2}]\\
G(x,2t-1) &\trm{if } t\in [\frac{1}{2},1]
\end{array} \right..\eeqs Then $H$ is a homotopy interpolating $f_0$ and $f_2$, and \[\cs(H)=(-1)^{\bl-1}\int_I\ch(H)=(-1)^{\bl-1}\int_I\ch(F)+(-1)^{\bl-1}\int_I\ch(G)\] is an element of $\Om^*_{\trm{exact}}(X)$. Hence $f_0\sim_{\cs}f_2$.
\end{proof}
\begin{nta} We denote by $[f]_{\cs}$ the cs-equivalence class of a map $f:X\ra \cF_\bl$.
\end{nta}
\begin{dfn} The Tradler-Wilson-Zeinalian differential $K$-theory of $X$ is the set $$\widehat K^\bl_{TWZ}(X):=\{[f]_{\cs}:f:X\ra \cF\trm{ is a smooth map.}\}$$ endowed with a structure of abelian group induced by $m_\bl$.
\end{dfn}

\begin{thm}\label{THM.theorem.2}
The assignment \beqs \Phi:\widehat K^\bl_{TWZ}(X) &\ra \check K^\bl(X)\\
[f]_{\cs} &\mapsto [(f,0,\ch(f))]
\eeqs is an isomorphism of abelian groups that is natural in $X$.
\end{thm}
\begin{proof}
See Section \ref{ProofofTheorem2}.
\end{proof}


\section{Proof of Theorem \ref{THM.theorem.1}}\label{ProofofTheorem1}
\subsection{Well-definedness}\label{SEC.welldef.smoothHS.to.HS} Suppose two triples
$(f_0,h_0,\om_0)$ and $(f_1,h_1,\om_1)$ are equivalent; i.e., $\om_0=\om_1$ and there exists an interpolating triple $(F,H,p^*\om)$ satisfying $dH=F^*c-p^*\om$. Integrating both sides we obtain: $\de \int H=F^*\int c-\int p^*\om$, and hence $(F,\int H,p^*\om)$ is a triple that interpolates between $(f_0,\int h_0,\om_0)$ and $(f_1,\int h_1,\om_1)$.

\subsection{Injectivity}\label{sec.injectivity} For any two triples
$(f_0,h_0,\om_0)$ and $(f_1,h_1,\om_1)$, we assume that $(f_0,\int h_0,\om_0)$ and $(f_1,\int h_1,\om_1)$ are equivalent, i.e., $\om_0=\om_1$, and there exists an interpolating triple $(F,H,p^*\om)$ with the triple relation \beq\label{EQN.triple.rel.XxI}
\de H=F^*\int c-\int p^*\om.
\eeq

We pause our proof and prove a lemma that is necessary in proving injectivity. Recall that, given $f_0$, $f_1$, and $F$ be as in the above paragraph, there exists a homotopy $G:X\times I\times I\ra \cF_\bl$ between $F$ and a smooth approximation $\wbar F$ of $F$, which fixes the ends: $\wbar F(x,0)=F(x,0)=f_0(x)$ and $\wbar F(x,1)=F(x,1)=f_1(x)$. By a smooth approximation $\wbar F$ we mean a smooth map $\wbar F$ that is homotopic to a continuous map $F$ by a continuous homotopy $G$.

\begin{lem} \label{LEM.difference.of.chern}
$F^*\int c-\wbar F^*\int c=\de K$ for some $K$, where $K$ is defined in the proof.
\end{lem}
\begin{proof}
Let $G$ be as in the above paragraph. For $G^*\int c\in C^*(X\times I^2)$, we take the slant product with $I$ and then take the exterior derivative. By the derivation formula for the slant product, we have:
\beqs \de\left(G^*\int c\Big/I\right)&=\left(\de G^*\int c\right)\Big/I+(-1)^{\left|G^*\int c\right|+|I|}G^*\int c\Big/\ptl I\\
&=(-1)^{\bl}\left(F^*\int c-\wbar F^*\int c\right).\eeqs We denote $(-1)^{\bl}G^*\int c\Big/I$ by $K$. Note that $K|_{M\times\{0,1\}}=0$.
\end{proof}

Now we resume our proof. By Lemma \ref{LEM.difference.of.chern}, equation \eqref{EQN.triple.rel.XxI} can be written as $$\de(H-K)=\wbar F^*\int c-\int p^*\om_0.$$

Now consider a differential form \beq\label{EQN.triple.diff.form.correction}
\wbar F^*c-p^*\om_0-d(\f{H}) \eeq where $\f{H}:=(1-t)h_0+t h_1$. This is a closed form on $X\times I$, that vanishes on $X\times \{0,1\}$. Hence $\wbar F^*c-p^*\om_0-d(\f{H})\in \Om_{\trm{cl}}^*(X\times I, X\times \{0,1\}).$ If we integrate \eqref{EQN.triple.diff.form.correction}, we obtain $$\int\wbar F^*c-\int p^*\om_0-\int d(\f{H})=\de(H-K)-\de\int\f{H}.$$ By the relative de Rham theorem (see Appendix \ref{APP.A}), it follows that $\wbar F^*c-p^*\om_0-d(\f{H})=d'\xi$, for some $\xi\in\Om^{*-1}(X\times I)$ where $\xi\equiv 0$ on $X\times \{0,1\}$. (Here $d'$ is the differential of the relative complex --- see Appendix \ref{APP.A}.) We thus have an equality $\wbar F^*c-p^*\om_0=d(\xi+\f{H})$ in $X\times I$, which is the triple relation for $(\wbar F,\xi+\f{H},p^*\om_0)$. This triple interpolates between $(f_0,h_0,\om_0)$ and $(f_1,h_1,\om_1)$.

\subsection{Surjectivity}\label{SEC.surj.smoothhs.to.hs} We consider the special case that the classifying map $f:X\ra \cF_\bl$ is given by a smooth map and then the general case.

\textbf{Case I:} We choose an element $[(f,h,\om)]\in \widehat{K}^\bl(X)$, with the property that there exists a representative $(f,h,\om)$ with $f$ smooth. Consider the triple relation for the representative $(f,h,\om)$ $$\de h=f^*\int c-\int\om.$$
By de Rham theorem, $f^*c-\om=d\xi$ for some differential form $\xi$ on $X$, and hence $\de h=\de\int \xi$. Since $h-\int\xi$ represents a cohomology class, there exists $\eta\in\Om_{\text{cl}}^*(X)$ such that $[h-\int\xi]=[\int\eta]$ or equivalently, $h-\int\xi=\int\eta+\de\mu$ for some cochain $\mu$. We write $\zeta:=\xi+\eta$

Now we consider the triple $(f,\zeta,\om)$. The map \eqref{EQN.map.in.theorem.1} takes this triple to $(f,\int \zeta,\om)$. We claim that $(f,\int \zeta,\om)$ and $(f,\int \zeta+\de\mu,\om)$ are equivalent. The following lemma is standard but we give a proof for sake of completeness.

\begin{lem}\label{LEM.cohomologous.cocycle.on.XtimesI} If $\al$, $\be\in Z^n(X;\cR)$ such that $\al-\be=\de k$ for some $k\in C^{n-1}(X;\cR)$, then there exists $L\in Z^n(X\times I;\cR)$ such that $\psi_0^*L=\al$ and $\psi_1^*L=\be$. (Recall Notation \ref{NTA.1}.)
\end{lem}
\begin{proof} Consider an interval $I$ as a CW complex with two $0$-cells $a$, $b$, and a $1$-cell $e$. We define cochains $e^*\in C^1(I;\cR)$ by $e\mapsto 1$, $a^*\in C^0(I;\cR)$ by $a\mapsto 1$ and $b\mapsto 0$, and $b^*\in C^0(I;\cR)$ by $a\mapsto 0$ and $b\mapsto 1$. Then, for any $0$-cell $v$ in $X$, $$(\psi_0^*p_2^*)a^*(v)=a^*({p_2}_*{\psi_0}_*v)=a^*(a)=1,$$ and similarly,\beqs
(\psi_0^*p_2^*)b^*(v)&=b^*(a)=0\\
(\psi_1^*p_2^*)a^*(v)&=a^*(b)=0\\
(\psi_1^*p_2^*)b^*(v)&=b^*(b)=1.
\eeqs Also, for any $1$-cell $\si$ in $X$, $$(\psi_0^*p_2^*)e^*(\si)=e^*({p_2}_*{\psi_0}_*(\si))=e^*(0)=0,$$ and similarly, $(\psi_1^*p_2^*)e^*(\si)=0$. We also have:\beqs
\de a^*(e) &= a^*\ptl(e) = a^*(b-a)=-1\quad\Leftrightarrow\quad \de a^*=-e^*\\
\de b^*(e) &= b^*\ptl(e) = b^*(b-a)=1\quad\Leftrightarrow\quad \de b^*=e^*.
\eeqs Since there is no $2$-cell in $I$, we also have $\de e^*=0$.

Now we take $L:=p_1^*\al\cup p_2^*a^*+p_1^*\be \cup p_2^*b^*+(-1)^np_1^*k\cup p_2^*e^*$. We see that:\beqs
\de L&=(-1)^np_1^*\al\cup \de p_2^*a^*+(-1)^np_1^*\be \cup \de p_2^*b^*+(-1)^n\de p_1^*k\cup p_2^*e^*+0\\ &= (-1)^n\left(-p_1^*\al\cup p_2^*e^*+p_1^*\be \cup p_2^*e^*+\de p_1^*k\cup p_2^*e^*\right)=0,\trm{since $\de k=\al-\be$,}
\eeqs and also: \beqs \psi_0^*L&=(\psi_0^*p_1^*)\al\cup (\psi_0^*p_2^*)a^*+(\psi_0^*p_1^*)\be \cup (\psi_0^*p_2^*)b^*+(-1)^n(\psi_0^*p_1^*)k\cup (\psi_0^*p_2^*)e^*\\
&=\al\cup 1+\be \cup 0+0=\al,\\
\psi_1^*L&=(\psi_1^*p_1^*)\al\cup (\psi_1^*p_2^*)a^*+(\psi_1^*p_1^*)\be \cup (\psi_1^*p_2^*)b^*+(-1)^n(\psi_1^*p_1^*)k\cup (\psi_1^*p_2^*)e^*\\
&=\al\cup 0+\be \cup 1+0=\be.
\eeqs
\end{proof}

We continue proof of surjectivity. By Lemma \ref{LEM.cohomologous.cocycle.on.XtimesI}, there exists a cocycle $L$ on $X\times I$ such that $\psi_0^*L=0$ and $\psi_1^*L=\de\mu$. Using this $L$, we form the triple $(f\circ p,p^*\int\zeta+L,p^*\om)$. This triple restricts to $(f,\int \zeta,\om)$ and $(f,\int \zeta+\de\mu,\om)$ at each end. Furthermore, $$\de \left(p^*\int\zeta+L\right)=p^*\de \int\xi= p^*\left(f^*\int c-\int \om\right)=(f\circ p)^*\int c-\int p^*\om,$$ where in the first equality, we used the fact that $\eta$ is a closed form. Hence the claim. Therefore, for any $[(f,h,\om)]\in \widehat K^\bl(X)$ with smooth $f$, there is an element $[(f,\zeta,\om)]\in \check K^\bl(X)$ in the preimage.

\textbf{Case II:} We consider a triple whose classifying map is not necessarily smooth. Given any $[(f,h,\om)]\in \widehat K^\bl(X)$, it suffices to show that there exists a triple $(\wbar f,h',\om)$, with a smooth classifying map $\wbar f$, equivalent to $(f,h,\om)$. We consider a smooth approximation $\wbar f$ of $f$ satisfying that $\wbar f$ and $f$ are homotopic through a homotopy $g$ with $\wbar f=g(-,1)$.

\begin{lem}\label{LEM.triple.homotopy.effect} Let $(f_0,h,\om)$ be a triple representing an element of $\widehat{K}^\bl(X)$. Suppose $f_0$ is homotopic to $f_1$ via a homotopy $F$. Then the triple $(f_0,h,\om)$ is equivalent to $\left(f_1,h+(-1)^{|c|+1}F^*\int c \Big/ I,\om\right)$.
\end{lem}
\begin{proof} We take an interpolating triple $\left(F,p^*h+(-1)^{|c|+1}G^*\int c\Big/I,p^*\om\right)$, where $G$ is a homotopy between $F$ and $f_0\circ p$ defined by \beq\label{EQN.htpy.G} G: X\times I\times I &\ra \cF_\bl\\
(x,t,s)&\mapsto G(x,t,s):=\left\{ \begin{array}{ll}
F(x,t)\quad\trm{if } t\leq s\\
F(x,s) \quad\trm{if } t\geq s
\end{array} \right..\eeq In particular, $G(x,t,0)=F(x,0)=f_0\circ p$ and $G(x,t,1)=F(x,t)=F$.

Let $\wtl{\psi_s}:X\times I \ra X\times I\times I$ be a $s$-slice map defined by $(x,t)\mapsto (x,t,s)$.

Since \beqs
\wtl \psi_0^*\left(G^*\int c\Big/I\right)&=(G\circ\wtl{\psi_0})^*\int c\Big/I=(f_0\circ p)^*\int c\Big/I=0\\
\wtl \psi_1^*\left(G^*\int c\Big/I\right)&=(G\circ\wtl{\psi_1})^*\int c\Big/I=F^*\int c\Big/I,
\eeqs the triple $\left(F,p^*h+(-1)^{|c|+1}G^*\int c\Big/I,p^*\om\right)$ interpolates the between given two triples. We verify the triple relation:

\beqs
&\de\left(p^*h+(-1)^{|c|+1}G^*\int c\Big/I \right)\\&\quad=p^*\de h +\de\left((-1)^{|c|+1}G^*\int c\Big/I\right)\\
&\quad=p^*(f_0^*\int c-\int\om)+(-1)^{|c|+1}\de G^*\int c\Big/I +G^*\int c\Big/\ptl I\\
&\quad=p^*f_0^*\int c-\int p^*\om+\left(\wtl{\psi_1}^*G^*\int c-\wtl{\psi_0}^*G^*\int c \right)\\
&\quad=p^*f_0^*\int c-\int p^*\om+\left(F^*\int c-(f_0\circ p)^*\int c \right)
=F^*\int c-\int p^*\om.
\eeqs
\end{proof}

Therefore, by Lemma \ref{LEM.triple.homotopy.effect}, the triple $\left(\wbar f,h+(-1)^{|c|+1}F^*\int c \Big/ I,\om\right)$ is equivalent to $(f,h,\om)$. Now a preimage of $[(\wbar f,h',\om)]$ can be found, by Case I.

\subsection{Group homomorphism and naturality}\label{SEC.GroupHomo.Thm1} We first prove that \eqref{eqn.HS.triple.sum} and \eqref{eqn.smooth.HS.triple.sum} is an abelian group operation as claimed in Lemmas \ref{LEM.group.structure.on.HS} and \ref{LEM.group.structure.on.smoothHS}, respectively. Note that $(\cI_\bl,0,0)$ is the identity. The existence of inverses will follow from similar arguments. We prove associativity presently.

Consider any three triples $(f_1,h_1,\om_1)$, $(f_2,h_2,\om_2)$, and $(f_3,h_3,\om_3)$ in $\widehat K^\bl(X)$. By the argument in Case II in Section \ref{SEC.surj.smoothhs.to.hs}, we may assume that $f_1$, $f_2$, and $f_3$ are smooth. Consider the following triples
\beq\label{EQN.Associativity.triples} &(m_\bl(f_1,m_\bl(f_2,f_3)),h_1+h_2+h_3,\om_1+\om_2+\om_3)\\
&(m_\bl(m_\bl(f_1,f_2),f_3)),h_1+h_2+h_3,\om_1+\om_2+\om_3)\eeq


By \cite[Lemmas 3.24 (2) and 3.9]{TWZ2} for $m_0$ and $m_1$, respectively, two maps $m_\bl(f_1,m_\bl(f_2,f_3))$ and $m_\bl(m_\bl(f_1,f_2),f_3)$ are cs-equivalent such that $\cs(\Ga_\bl)=0$ for some homotopy $\Ga_\bl$ between these maps. Smoothness of the homotopy $\Ga_\bl$ also follows from \cite[Lemmas 3.24 (2) and 3.9]{TWZ2}. From Lemmas \ref{LEM.triple.homotopy.effect} (using the homotopy $\Ga_\bl$) and \ref{LEM.cohomologous.cocycle.on.XtimesI}, it follows that the triples in \eqref{EQN.Associativity.triples} are equivalent.

Now suppose any three triples $(f_1,h_1,\om_1)$, $(f_2,h_2,\om_2)$, and $(f_3,h_3,\om_3)$ are in $\check K^\bl(X)$. Again $m_\bl(f_1,m_\bl(f_2,f_3))\sim_{\text{cs}} m_\bl(m_\bl(f_1,f_2),f_3)$ by the same reason, and triples again of the form \eqref{EQN.Associativity.triples} are equivalent by a similar argument in Section \ref{SEC.welldef.smoothHS.TWZ} below and the fact that $\cs(\Ga_\bl)=0$.

It is readily seen that the map \eqref{EQN.map.in.theorem.1} is a group homomorphism. It is natural in $X$ by the change of variables formula.


\section{Proof of Theorem \ref{THM.theorem.2}}\label{ProofofTheorem2}

\subsection{Well-definedness}\label{SEC.welldef.smoothHS.TWZ} Suppose $f_0\sim_{\cs}f_1$ through a homotopy $F$. We have to show that two triples $(f_0,0,\ch(f_0))$ and $(f_1,0,\ch(f_1))$ in $\check K^\bl(X)$ are equivalent. Since $\cs(F)$ is exact, it follows that $\ch(f_0)=\ch(f_1)$. We define an interpolating triple by $(F,\cs(G),p^*\ch(f_0))$, where $G$ is a homotopy between $f_0\circ p$ and $F$ defined in \eqref{EQN.htpy.G}. We have the triple relation $$ d\cs(G)=\ch(F)-\ch(f_0\circ p),$$ and the triple $(F,\cs(G),p^*\ch(f_0))$ becomes $(f_0,\cs(f_0\circ p\circ \wtl\psi_0),\ch(f_0))$ (resp. $(f_1,\cs(F),\ch(f_0))$) when it is restricted to $X\times\{0\}$ (resp. $X\times \{1\}$). We claim that triples $(f_1,\cs(F),\ch(f_0))$ and $(f_1,0,\ch(f_0))$ are equivalent. This can be easily verified by applying the following Lemma.

\begin{lem}\label{LEM.cohomologous.closedform.on.XtimesI} If $\al$, $\be\in \Om^n_{\trm{cl}}(X)$ are such that $\al-\be=d\ga$ for some $\ga\in\Om^{n-1}(X)$, then there exists $\xi\in\Om^n_{\trm{cl}}(X\times I)$ such that $\psi_0^*\xi=\al$ and $\psi_1^*\xi=\be$.
\end{lem}
\begin{proof} Set $\xi:=(1-t)p^*\al+tp^*\be-dt\we p^*\ga\in \Om^n(X\times I)$. Then $d\xi=-dt\we p^*\al+dt\we p^*\be+dt\we p^*d\ga=0$, $\psi_0^*\xi=\al$, and $\psi_1^*\xi=\be$.
\end{proof}

Since $\cs(F)$ is exact, we may write $\cs(F):=d\mu$. We apply Lemma \ref{LEM.cohomologous.closedform.on.XtimesI} with $\al=d\mu$ and $\be=0$. More explicitly, the interpolating triple between $(f_1,\cs(F),\ch(f_0))$ and $(f_1,0,\ch(f_0))$ is $(f_1\circ p,\xi,p^*\ch(f_1))$ where $$\xi:=(1-t)p^*d\mu-dt\we p^*\mu.$$ We see that $d\xi=-dt\we p^*d\mu+dt\we p^*d\mu=0$, $\psi^*_0\xi=d\mu$, and $\psi^*_1\xi=0$. The triple relation is easily verified: $d\xi=0=\ch(f_1\circ p)-p^*\ch(f_1)$. Thus the map is well-defined.

\subsection{Injectivity} Suppose two triples $(f_0,0,\ch(f_0))$ and $(f_1,0,\ch(f_1))$ in $\check K^\bl(X)$ are equivalent. i.e. $\ch(f_0)=\ch(f_1)$ and there exists a homotopy $F$ between $f_0$ and $f_1$, such that $$dH=\ch(F)-p^*\ch(f_0)$$ for some differential form $H\in\Om^{\bl-1}(X\times I)$. Integrating both sides along $I$, we get \beqs \int_I \ch(F)-0 &=\int_I \ch(F)-\int_I p^*\ch(f_0) = \int_I dH+0 \\& =\int_I dH +(-1)^{|H|-1}\int_{\ptl I}H = d\int_I H. \eeqs This shows that $\cs(F)$ is exact, and hence $f_0\sim_{\cs}f_1$.

\subsection{Surjectivity} We need two lemmas. The following lemma follows from \cite{TWZ2}.
\begin{lem}[Strong Venice Lemma] \label{LEM.strongvenice} Given a differential form $h\in\Om^{\bl-1}(X)$ and a smooth map $f_1:X\ra \cF_\bl$, there exists a smooth map $f_0:X\ra \cF_\bl$ and a smooth homotopy $F:X\times I\ra \cF_\bl$ between $f_0$ and $f_1$ such that $\cs(F)=h$.
\end{lem}
\begin{proof} Given such a differential form $h$, Theorem 3.17 (2) of \cite{TWZ2} shows that there exists a homotopy $G:X\times I\ra \cF_\bl$ such that $G(x,1)=\cI_\bl$ and $\cs(G)=h$. (See Notation \ref{NTA.2} for the definition of $\cI_\bl$.)

Accordingly, define a homotopy $F:X\times I\ra \cF_\bl$ by $$F(x,t):= m_\bl(G(x,t),p(f_1(x))).$$ Note that $F(x,1)=m_\bl(G(x,1),p(f_1(x)))=f_1(x)$ and $$\cs(F)=(-1)^{\bl-1}\int_I\ch(G)+(-1)^{\bl-1}\int_I p\circ f_1=\cs(G)+0=h.$$ At $t=0$, $F(x,0)=m_\bl(G(x,0),p(f_1(x)))$ which we denote by $f_0(x)$.
\end{proof}

\begin{lem}\label{LEM.triple.homotopy.effect.smooth} Let $(f_0,h,\om)$ be a triple representing an element of $\check K^\bl(X)$. Suppose $f_0$ is homotopic to $f_1$ via a homotopy $F$. Then the triple $(f_1,h+\cs(F),\om)$ is equivalent to $(f_0,h,\om)$.
\end{lem}
\begin{proof} We choose $(F,p^*h+\cs(G),p^*\om)$, where $G$ is a homotopy between $F$ and $f_0\circ p$ defined in \eqref{EQN.htpy.G}. We verify that this triple interpolates between $(f_0,h,\om)$ and $(f_1,h+\cs(F),\om)$. First, we have \beqs
\wtl \psi_0^*(F,p^*h+\cs(G),p^*\om)&=(f_0,h+\cs(G\circ \wtl \psi_0),\wtl \psi_0^*p^*\om)=(f_0,h,\om),\\
\wtl \psi_1^*(F,p^*h+\cs(G),p^*\om)&=(f_1,h+\cs(G\circ \wtl \psi_1),\wtl \psi_1^*p^*\om)=(f_1,h+\cs(F),\om),
\eeqs where $\wtl \psi_s$ is as defined in the proof of Lemma \ref{LEM.triple.homotopy.effect}.

The triple condition then follows from
$$d\left(p^*h+\cs(G)\right)=p^*f_0^*c-p^*\om+F^*c-p^*f_0^*c=F^*c-p^*\om.$$
\end{proof}

We now prove surjectivity. Take any representative $(f_1,h,\om)$ of any element in $\check K^\bl(X)$. Applying Lemma \ref{LEM.strongvenice} with $h$ and $f_1$, we may write $h=\cs(F)$ where $F$ is a homotopy between $f_0$ and $f_1$ for some $f_0$. Note that $\ch(f_0)=\om$, because $$\ch(f_1)-\ch(f_0)=dh=\ch(f_1)-\om.$$ By Lemma \ref{LEM.triple.homotopy.effect.smooth}, the triple $(f_0,0,\om)$ is equivalent to the triple $(f_1,\cs(F),\om)$. Since $(f_0,0,\om)$ is a representative of the image of $[f_0]\in \wht K^{\bl}_{\trm{TWZ}}(X)$, $[f_0]$ is a preimage of $(f_1,h,\om)\in \check K^\bl(X)$.

\subsection{Group homomorphism and naturality}\label{SEC.GroupHomo.Thm2} The given map is a group homomorphism since $\ch(m(f_0,f_1))=\ch(f_0)+\ch(f_1)$. It is natural in $X$ by the naturality of $\ch$.

\appendix
\section{Relative de Rham theorem} \label{APP.A}
We state and prove a variant of de Rham theorem for relative cohomology groups. This result is certainly well-known and classical, but we did not find a reference. Throughout this appendix, let $X$ be a smooth manifold, $Y$ a closed submanifold, and $\imath:Y\embed X$ a smooth embedding.

\begin{dfn} The \textbf{relative de Rham complex} $(\Om^\bl(X,Y),d'_\bl)$ is defined by $\Om^k(X,Y):=\ker \imath_k^*$ and $d'_k:=d|_{\ker \imath_k^*}$ for each $k\geq 0$, where $\imath_k^*:\Om^k(X)\ra \Om^k(Y)$ is the restriction map.
\end{dfn}

Since $\imath^*$ commutes with $d$, the image of $d_k'$ is contained in $\Om^{k+1}(X,Y)$, hence the pair $(\Om^{\bl}(X,Y),d_\bl')$ is a complex. Also note that the kernel of $d'_k$ is $\Om_{\trm{cl}}^k(X)\cap \ker i_k^*$.

\begin{dfn} The degree $k$ \textbf{relative de Rham cohomology groups} of the pair $(X,Y)$ is defined by $$H_{\trm{dR}}^{k}(X,Y):=\ker d'_k \big/ \Ima d'_{k-1}, \quad k\in \Z.$$
\end{dfn}

\begin{lem}(1) In the following diagram, the rows are exact, and the squares are commutative.
\[\xymatrix{
0\ar[r] & \Om^k(X,Y)\ar[d]^{d'_k}\ar[r]^{\embed} & \Om^k(X)\ar[d]^{d_k}\ar[r]^{\imath_k^*} & \Om^k(Y)\ar[d]^{d_k}\ar[r] & 0\\
0\ar[r] & \Om^{k+1}(X,Y)\ar[r]^{\embed} & \Om^{k+1}(X)\ar[r]^{\imath_{k+1}^*} & \Om^{k+1}(Y)\ar[r] & 0
}\]

\noindent (2) The following sequence of cohomology groups is exact. $$0\ra H_{\trm{dR}}^0(X,Y)\ra H_{\trm{dR}}^0(X)\ra H_{\trm{dR}}^0(Y)\srl{\de_{\trm{dR}}}\lra H_{\trm{dR}}^1(X,Y)\ra H_{\trm{dR}}^1(X)\ra\cdots.$$
\end{lem}
\begin{proof}
(1) The surjectivity of the map $\imath_k^*:\Om^k(Y)\ra \Om^k(X)$ for $k\in\Z$ follows from Whitney's embedding theorem. The diagram is commutative because pull-back commutes with $d$. (2) follows from the snake lemma.
\end{proof}
\begin{prp}[Relative de Rham Theorem] The assignment \beqs
\int:H_{\trm{dR}}^k(X,Y)&\ra H^k(X,Y;\cR)\\
[\om] &\mapsto \left[\int\om\right],
\eeqs is a natural isomorphism of groups for each $k\in\Z$.
\end{prp}
\begin{proof} We use the $5$-lemma. Consider the following diagram. \[\xymatrix{
H_{\trm{dR}}^{k-1}(X)\ar[d]_{\isom}^{\int} \ar[r]^{\imath_{k-1}^*} & H_{\trm{dR}}^{k-1}(Y)\ar[d]_{\isom}^{\int} \ar[r]^{\de_{\trm{dR}}} & H_{\trm{dR}}^{k}(X,Y)\ar[d] \ar[r] & H_{\trm{dR}}^{k}(X)\ar[d]_{\isom}^{\int} \ar[r]^{\imath_{k}^*} & H_{\trm{dR}}^{k}(Y)\ar[d]_{\isom}^{\int}\\
H^{k}(X)\ar[r]^{\imath_{k-1}^*} & H^{k}(Y)\ar[r]^\de & H^{k}(X,Y)\ar[r] & H^{k}(X)\ar[r]^{\imath_{k}^*} & H^{k}(Y)}\]
The first and the last squares are commutative by naturality of de Rham theorem. The third square commutes by Stokes' formula. We verify the commutativity of the second square. Take any $[\te]\in H_{\trm{dR}}^{k-1}(Y)$. There exists $\eta\in\Om^{k-1}(X)$ whose restriction to $Y$ is $\te$. The restriction of the differential form $d\eta\in \Om^k(X)$ to $Y$ is identically zero. Hence $d\eta$ is a representative of the cohomology class $\de_{\trm{dR}}([\te])\in H_{\trm{dR}}^{k}(X,Y)$. Applying the vertical map, we obtain $[\int d\eta]$. Now we apply the vertical map to $[\te]$ and then the connecting map. This gives $\de[\int\te]\in H^{k}(X,Y;\cR)$. Since $\int \eta$ is a cochain in $X$ that restricts to $\int\te$ in $Y$, the $k$-cochain $\int d\eta$ represents the smooth singular relative cohomology class $\de[\int\te]$, since $\imath_k^*(\int d\eta)$ is vanishing.
\end{proof}

\end{document}